\documentclass{article}

\usepackage{amssymb,amsthm,amsmath}
\usepackage{enumerate}

\newtheorem{thm}{Theorem}
\newtheorem{lem}[thm]{Lemma}

\theoremstyle{definition}
\newtheorem{defn}[thm]{Definition}

\newenvironment{principle}[3][]{\par\ensuremath{\mathsf{#2}_{#1}} ({#3}):\em}{\par}

\title{Nielsen-Schreier implies the finite Axiom of Choice}
\author{Philipp Kleppmann\footnote{Department of Pure Mathematics and Mathematical Statistics, Centre for Mathematical Sciences, Wilberforce Road, Cambridge, CB3 0WB, UK. Email: P.Kleppmann@dpmms.cam.ac.uk.}}

\begin{document}

\maketitle

\begin{abstract}
We present a new proof that the statement 'every sugbroup of a free group is free' implies the Axiom of Choice for finite sets.
\end{abstract}

\section{Introduction}

In 1921, Nielsen \cite{Nielsen1921NS} proved that every subgroup of a finitely generated free group is free. This result was generalised to arbitrary free groups by Schreier \cite{Schreier1927Untergruppen} in 1927, giving us the following result.

\begin{principle}{NS}{Nielsen-Schreier}
If $F$ is a free group and $K\leq F$ is a subgroup, then $K$ is a free group.
\end{principle}

Since every proof of $\mathsf{NS}$ uses the Axiom of Choice, it is natural to ask whether it is equivalent to the Axiom of Choice. The first step was made by L\"auchli \cite{Lauchli1962Auswahlaxiom}, who showed that $\mathsf{NS}$ cannot be proved in $\mathsf{ZF}$ set theory with atoms. Jech and Sochor's embedding theorem \cite{Jech2008Choice} allows this result to be transferred to standard $\mathsf{ZF}$ set theory. It was improved in 1985 by Howard \cite{Howard1985Subgroups}, who showed that $\mathsf{NS}$ implies $\mathsf{AC}_{fin}$, the Axiom of Choice for finite sets:

\begin{principle}[fin]{\mathsf{AC}}{Axiom of Choice for finite sets}
Every set of non-empty finite sets has a choice function.
\end{principle}

Another Choice principle used in this article is the Axiom of Choice for pairs:

\begin{principle}[2]{\mathsf{AC}}{Axiom of Choice for pairs}
Every set of 2-element sets has a choice function.
\end{principle}

The purpose of this paper is to provide a new and shorter proof of Howard's result.

\section{Nielsen-Schreier implies $\mathsf{AC}_{fin}$}
\label{section: Nielsen-Schreier implies ACfin}

Before beginning the proof, we must fix some notation and terminology. If $X$ is a set, let $X^-=\{x^{-1}:x\in X\}$ be a set of formal inverses of $X$. It does not matter what the elements of $X^-$ are, as long as $X^-$ is disjoint from $X$. Members of $X^\pm=X\cup X^-$ are called $X$-\emph{letters}. Finite sequences $x_1\cdots x_n$ with $x_1,...,x_n\in X^\pm$ are $X$-\emph{words}. An $X$-word $x_1\cdots x_n$ is $X$-\emph{reduced} if $x_i\not=x_{i+1}^{-1}$ for $i=1,...,n-1$. If $\alpha$ is an $X$-word, the $X$-\emph{reduction} of $\alpha$ is the $X$-reduced $X$-word obtained by performing all possible cancellations within $\alpha$. For notational simplicity, we don't distinguish between $X$-words and their $X$-reductions. Reference to $X$ is omitted if $X$ is clear from the context.

If $G$ is a group and $S\subseteq G$, then $\langle S\rangle$ is the subgroup of $G$ \emph{generated by} $S$.

\begin{defn}
Let $X$ be a set. The \emph{free group on} $X$, written $F(X)$, consists of all reduced $X$-words. The group operation is concatenation followed by reduction, and the identity is the empty word ${\bf 1}$.

A group $G$ is \emph{free} if it is isomorphic to $F(X)$ for some $X\subseteq G$. If this is the case, $X$ is a \emph{basis} for $G$.
\end{defn}

The following proofs will start with a family $Y$ of non-empty sets and construct a choice function $c:Y\rightarrow\bigcup Y$. Without loss of generality, we assume that the members of $Y$ are pairwise disjoint. We then define $X=\bigcup Y$ to be the basis of the free group $F=F(X)$. With every $y\in Y$ we associate a function  $\sigma_y:F\rightarrow\mathbb{Z}$ which counts the number of occurrences of $y$-letters in words $\alpha\in F$ as follows.
\begin{quote}
Write $\alpha=x_1^{\epsilon_1}\cdots x_n^{\epsilon_n}$ as an $X$-reduced word with $x_1,...,x_n\in X$ and $\epsilon_1,...,\epsilon_n\in\{\pm1\}$. Then define
$$\sigma_y(\alpha)=|\{i:x_i\in y\land\epsilon_i=1\}|-|\{i:x_i\in y\land\epsilon_i=-1\}|.$$
It is easily checked that, for each $y\in Y$, $\sigma_y$ is a group homomorphism from the free group $F$ to the additive group of integers. \end{quote}

Before proving theorem \ref{theorem: NS implies ACfin} we handle a special case in lemma \ref{lemma: NS implies AC2}. Its proof serves as an introduction to ideas used in the proof of the main theorem.

\begin{lem}
\label{lemma: NS implies AC2}
$\mathsf{ZF}\vdash\mathsf{NS}\Rightarrow\mathsf{AC}_2$
\end{lem}
\begin{proof}
Let $Y$ be a family of 2-element sets. Without loss of generality, assume that the members of $Y$ are pairwise disjoint.

Let $X=\bigcup Y$, let $F=F(X)$ be the free group on $X$, and define the subgroup $K\leq F$ by
$$K=\langle\{wx^{-1}:(\exists y\in Y)w,x\in y\}\rangle.$$
By the Nielsen-Schreier theorem, $K$ has a basis $B$. Note that
\begin{equation}
\label{equation: K is in the kernel}
\sigma_y(\alpha)=0\text{ for all }y\in Y\text{ and all }\alpha\in K.
\end{equation}

We will construct a choice function for $Y$, i.e. a function $c:Y\rightarrow X$ satisfying $c(y)\in y$ for each $y\in Y$.

Let $y\in Y$. Define the function $s_y:y\rightarrow y$ to swap the two elements of $y$. For any choice of $x\in y$, $y=\{x,s_y(x)\}$. To simplify notation, we set $x_i=s_y^i(x)$ for all $i\in\mathbb{Z}$; hence $y=\{x_0,x_1\}$. Express $x_0x_1^{-1}$ and $x_1x_0^{-1}$ as reduced $B$-words:
\begin{eqnarray*}
x_0x_1^{-1}&=&b_{0,1}\cdots b_{0,l_0}\\
x_1x_0^{-1}&=&b_{1,1}\cdots b_{1,l_1},
\end{eqnarray*}
where $b_{i,j}\in B^\pm$ for all $i,j$. As $x_0x_1^{-1}=(x_1x_0^{-1})^{-1}$, it follows that $l_0=l_1=l$, say, and that 
\begin{equation}
\label{equation: cancellation in NS=>AC2}
b_{1,1}=b_{0,l}^{-1},...,b_{1,l}=b_{0,1}^{-1}.
\end{equation}

There are two cases:

\begin{enumerate}[(i)]
\item $l$ is odd:

Let $m=(l-1)/2$. The middle $B$-letter of $x_0x_1^{-1}$ is $b_{0,m+1}$, whereas the middle $B$-letter of $x_1x_0^{-1}$ is $b_{1,m+1}=b_{0,m+1}^{-1}$ by (\ref{equation: cancellation in NS=>AC2}). One of these two is in $B$, while the other is in $B^-$. Define $c(y)$ to be the unique element $x\in y$ such that the middle $B$-letter of $xs_y(x)^{-1}$ is a member of $B$.

\item $l$ is even:

Let $m=l/2$. The following two functions are the key to the proof.
\begin{eqnarray*}
f_y:&y\rightarrow K:&x_i\mapsto b_{i,1}\cdots b_{i,m}\\
g_y:&y\rightarrow F:&x\mapsto f_y(x)^{-1}\cdot x
\end{eqnarray*}

The idea of $f_y$ is to map $x_i$ to the 'first half' of $x_ix_{i+1}^{-1}$ in terms of the new basis $B$. $f_y(x)$ is intended to represent $x$ in $K$.

Using (\ref{equation: cancellation in NS=>AC2}), we obtain
\begin{equation}
\begin{split}
\label{equation: f is well-behaved}
f_y(x_i)f_y(x_{i+1})^{-1}&=b_{i,1}\cdots b_{i,m}b_{i+1,m}^{-1}\cdots b_{i+1,1}^{-1}\\
&=b_{i,1}\cdots b_{i,m}b_{i,m+1}\cdots b_{i,2m}\\
&=x_ix_{i+1}^{-1}.
\end{split}
\end{equation}

It follows that $g_y(x_0)=g_y(x_1)$. Hence the image of $y$ under $g_y$ has a single member, $\alpha_y$, say. Note that
\begin{equation}
\label{equation: sigma of alpha is 1}
\begin{split}
\sigma_y(\alpha_y)& =\sigma_y(g_y(x_0))\\
& =\sigma_y(f_y(x_0)^{-1}x_0)\\
& =\sigma_y(f_y(x_0)^{-1})+\sigma_y(x_0)\\
& =0+1\text{ using (\ref{equation: K is in the kernel}), }f_y(x_0)\in K\text{, and }x_0\in y
\end{split}
\end{equation}
is non-zero. This means that $\alpha_y$ mentions at least one $y$-letter. So we define $c(y)$ to be the $y$-letter which appears first in the $X$-reduction of $\alpha_y$.
\end{enumerate}

\end{proof}

We are now ready to prove the general case:

\begin{thm}
\label{theorem: NS implies ACfin}
$\mathsf{ZF}\vdash\mathsf{NS}\Rightarrow\mathsf{AC}_{fin}$.
\end{thm}
\begin{proof}
Let $Z$ be a family of non-empty finite sets. Without loss of generality, assume that the members of $Z$ are pairwise disjoint. We form a new family 
$$Y=\{y:y\not=\emptyset\land(\exists z\in Z)y\subseteq z\},$$
i.e. the closure of $Z$ under taking non-empty subsets. Since $Z\subseteq Y$, any choice function for $Y$ immediately gives a choice function for $Z$.

Let $X=\bigcup Y$, let $F=F(X)$ be the free group on $X$, and let $K\leq F$ be the subgroup defined by
$$K=\langle\{wx^{-1}:(\exists y\in Y)w,x\in y\}\rangle.$$
By the Nielsen-Schreier theorem, $K$ has a basis $B$.

For each $n<\omega$, let $Y^{(n)}=\{y\in Y:|y|=n\}$ and $Y^{(\leq n)}=\{y\in Y:|y|\leq n\}$. By induction on $n$, we will find a choice function $c_n$ on $Y^{(\leq n)}$ for each $2\leq n<\omega$. By construction, the $c_n$ will be nested, so that $\bigcup_{2\leq n<\omega}c_n$ is a choice function for $Y$.

A choice function $c_2$ on $Y^{(\leq2)}$ is guaranteed by lemma \ref{lemma: NS implies AC2}.

Assume that $n\geq3$ and that there is a choice function $c_{n-1}$ for $Y^{(\leq n-1)}$. For every $y\in Y^{(n)}$ we define a function $s_y$ by
$$s_y:y\rightarrow y:x\mapsto c_{n-1}(y\setminus\{x\}).$$
Note that, as $Y$ is closed under taking non-empty subsets, $y\setminus\{x\}\in Y^{(n-1)}$, so $c_{n-1}(y\setminus\{x\})$ is defined. There are four cases:

\begin{enumerate}[(i)]
\item $s_y$ is not a bijection:

In this case, $|\{s_y(x):x\in y\}|\leq n-1$, so defining
$$c_n(y)=c_{n-1}(\{s_y(x):x\in y\})$$
gives a choice for $y$.

\item $s_y$ is a bijection with at least two orbits\footnote{Thanks to Thomas Forster for suggesting a simplification of this part of the proof}:

Since there are at least two orbits, each orbit has size $\leq n-1$. Moreover, as $s_y(x)\not=x$ for all $x\in y$, the number of orbits is also $\leq n-1$. So choosing one point from each orbit, and then choosing one point from among the chosen points gives a single element of $y$. More specifically, if we write $orb(x)$ for the orbit of $x\in y$ under $s_y$, we define
$$c_n(y)=c_{n-1}(\{c_{n-1}(orb(x)):x\in y\}).$$

\item $s_y$ is a bijection with one orbit, and $n$ is even:

If $n$ is even, $s_y^2$ is a bijection with two orbits. Remembering that we are assuming $n\geq 3$, this gives us $\leq n-1$ orbits of size $\leq n-1$ each. A choice is made as in the previous case.

\item $s_y$ is a bijection with one orbit, and $n$ is odd:

Notice that, for any $x\in y$, $y=\{x,s_y(x),s_y^2(x),...,s_y^{n-1}(x)\}$. $s_y(x)$ may be viewed as the successor of $x$. For simplicity, we set $x_i=s_y^i(x)$ for $i\in\mathbb{Z}$, so that $y=\{x_0,x_1,...,x_{n-1}\}$. 

In order to further simplify our notation, we shall assume that the elements of $Y^{(n)}$ are pairwise disjoint. Of course, this is not possible when $Y$ is constructed as above. But replacing every $y\in Y^{(n)}$ with $y\times\{y\}$ makes no difference to the argument, so the proof carries over without any changes.

Recall the basis $B$ of the subgroup $K$ defined earlier in the proof. We may write
\begin{eqnarray*}
x_0x_1^{-1}&=&b_{0,1}\cdots b_{0,l_0}\\
x_1x_2^{-1}&=&b_{1,1}\cdots b_{1,l_1}\\
&...&\\
x_{n-1}x_0^{-1}&=&b_{n-1,1}\cdots b_{n-1,l_{n-1}}
\end{eqnarray*}
as reduced $B$-words, with $b_{i,j}\in B^\pm$ for all $i,j$. First, we make two simplifications:

\begin{enumerate}[(a)]
\item If it is \emph{not} the case that $l_0=...=l_{n-1}$, let $l=min\{l_i:i=0,...,n-1\}$. Then $\{x_i:l_i=l\}$ is a proper non-empty subset of $y$, and we define
$$c_n(y)=c_{n-1}(\{x_i:l_i=l\}).$$
From now on it is assumed that $l_0=...=l_{n-1}=l$, say.

\item Note that
$$(x_0x_1^{-1})(x_1x_2^{-1})\cdots(x_{n-1}x_0^{-1})={\bf 1},$$
i.e.
\begin{equation}
\label{equation: cyclic cancelling}
(b_{0,1}\cdots b_{0,l})(b_{1,1}\cdots b_{1,l})\cdots(b_{n-1,1}\cdots b_{n-1,l})={\bf 1}.
\end{equation}

For $i=0,...,n-1$, let $k_i$ be the number of $B$-cancellations in 
\begin{equation}
\label{equation: number of cancellations}
(b_{i,1}\cdots b_{i,l})(b_{i+1,1}\cdots b_{i+1,l}).
\end{equation}
If it is \emph{not} the case that $k_0=...=k_{n-1}$, let $k=min\{k_i:i=0,...,n-1\}$. Then $\{x_i:k_i=k\}$ is a proper non-empty subset of $y$, and we define
$$c_n(y)=c_{n-1}(\{x_i:k_i=k\}).$$
From now on it is assumed that $k_0=...=k_{n-1}=k$, say.
\end{enumerate}

As letters always cancel in pairs, (\ref{equation: cyclic cancelling}) implies that $nl$ is even.\footnote{I would like to thank John Truss and Benedikt L\"owe for finding an error in this proof and suggesting a solution.} Since we are assuming that $n$ is odd, it follows that $l$ is even. Define $m=l/2$, and note that $k\geq m$: if not, then complete cancellation in (\ref{equation: cyclic cancelling}) would not be possible. This allows us to define functions $f_y$ and $g_y$, as in the proof of lemma \ref{lemma: NS implies AC2}:
\begin{eqnarray*}
f_y:&y\rightarrow K:&x_i\mapsto b_{i,1}\cdots b_{i,m}\\
g_y:&y\rightarrow F:&x\mapsto f_y(x)^{-1}x.
\end{eqnarray*}
Since there are $k\geq m$ cancellations in (\ref{equation: number of cancellations}), we have $b_{i+1,1}=b_{i,l}^{-1},...,b_{i+1,m}=b_{i,l-m+1}^{-1}=b_{i,m+1}^{-1}$
for all $i$. By the same calculation as in (\ref{equation: f is well-behaved}), it follows that 
$$f_y(x_i)f_y(x_{i+1})^{-1}=x_ix_{i+1}^{-1}$$
for all $i$, and hence that $g_y(x_i)=g_y(x_{i+1})$ for all $i$. So $g_y:y\rightarrow F$ is a constant function, taking a single value $\alpha_y$, say. The same calculation as (\ref{equation: sigma of alpha is 1}) yields 
$$\sigma_y(\alpha_y)=1.$$
So we set $c_n(y)$ to be the first $y$-letter occurring in the $X$-reduction of $\alpha_y$.

\end{enumerate}

\end{proof}

Whether or not the Nielsen-Schreier theorem is equivalent to the Axiom of Choice still remains an open question. A positive answer might be obtainable by adapting the proof of theorem \ref{theorem: NS implies ACfin}. Finiteness of the sets was used to define the choice function recursively, splitting up in cases (i) -- (iv). Cases (i) -- (iii) were easily dealt with. Case (iv) gave us a cyclic ordering on the finite set -- enough structure to use the basis of the subgroup $K$ to choose a single element.

\bibliography{bibliography}{}
\bibliographystyle{plain}

\end{document}